\documentclass[reqno]{amsart}
\usepackage{amsmath,amssymb,cite}
\usepackage[mathscr]{euscript}
\usepackage{graphics,epsfig,color}

\theoremstyle{plain}
\newtheorem{theorem}{Theorem}[section]
\newtheorem{proposition}[theorem]{Proposition}
\newtheorem{lemma}[theorem]{Lemma}

\theoremstyle{definition}
\newtheorem{assumption}[theorem]{Assumption}
\theoremstyle{remark}
\newtheorem{remark}{Remark}[section]

\numberwithin{equation}{section}
\numberwithin{theorem}{section}

\newcommand{\mc}[1]{{\mathcal #1}}

\newcommand{\bs}[1]{{\boldsymbol #1}}
\newcommand{\bb}[1]{{\mathbb #1}}

\newcommand{\rmd}{\mathrm{d}}
\newcommand{\eps}{\varepsilon}

\def\const{({\rm const.})\,}

\newcommand{\supp}{\mathop{\rm supp}\nolimits}
\renewcommand{\div}{\mathop{\rm div}\nolimits}

\title[Concentrated vortex rings]{Global time evolution of concentrated vortex rings}

\author[P.\ Butt\`a]{Paolo Butt\`a}
\address{Dipartimento di Matematica\\
Sapienza Universit\`a di Roma\\
P.le Aldo Moro 5, 00185 Roma\\
Italy}

\email{butta@mat.uniroma1.it}

\author[G.\ Cavallaro]{Guido Cavallaro}
\address{Dipartimento di Matematica\\
Sapienza Universit\`a di Roma\\
P.le Aldo Moro 5, 00185 Roma\\
Italy}

\email{cavallar@mat.uniroma1.it}

\author[C.\ Marchioro]{Carlo Marchioro}
\address{Dipartimento di Matematica\\
Sapienza Universit\`a di Roma\\
P.le Aldo Moro 5, 00185 Roma\\
Italy\\
and\\
International Research Center M\&MOCS\\ 
Universit\`a di L'Aquila\\
Palazzo Caetani\\
04012 Cisterna di Latina (LT)\\
Italy}

\email{marchior@mat.uniroma1.it}

\keywords{Incompressible Euler flow, vortex rings.}

\subjclass{
76B47, 
37N10. 
}

\date{}

\begin{document}

\begin{abstract}
We study the time evolution of an incompressible fluid with axial symmetry without swirl, assuming initial data such that the initial vorticity is very concentrated inside $N$ small disjoint rings of thickness $\eps$ and vorticity mass of the order of $|\log\eps|^{ -1}$.  When $\eps \to 0$ we show that the motion of each vortex ring converges to a simple translation with constant speed (depending on the single ring) along the symmetry axis.  We obtain a sharp localization of the vorticity support at time $t$ in the radial direction, whereas we state only a concentration property in the axial direction. This is obtained for arbitrary (but fixed) intervals of time. This study is the completion of a previous paper \cite{BuM2}, where a sharp localization of the vorticity support was obtained both along the radial and axial directions, but the convergence for $\eps\to 0$ worked only for short times.
\end{abstract}

\maketitle

\section{Introduction and main result}
\label{sec:1}

We study the time evolution of an incompressible non viscous fluid in the whole space $\bb R^3$, with an axial symmetry without swirl when the vorticity is sharply concentrated on $N$ annuli of radii $r_i\approx r_0$ and thickness $\eps$. In particular, we consider the limit $\eps \to 0$. In a previous paper of some years ago a similar problem \cite{BCM00} was investigated for a vortex alone, showing that it translates with a constant speed. Recently, in \cite{BuM2}, the analysis has been extended to the case of $N$ vortices, also getting a stronger localization property, but restricted to the case of short but positive time. In the present paper, we study the problem for any time.

The motion of an incompressible inviscid fluid is governed by the Euler equations, that for a fluid of unitary density in three dimension with velocity $\bs u = \bs u(\bs\xi,t)$ decaying at infinity read 
\begin{equation}
\label{vorteq}
\partial_t \bs\omega + (\bs u\cdot \nabla) \bs\omega  = (\bs \omega\cdot \nabla) \bs u  \,, \qquad 
\end{equation}
\begin{equation}
\label{u-vort}
\bs u(\bs\xi,t) = - \frac{1}{4\pi} \int_{\bb R^3}\! \rmd \bs\eta \, \frac{(\bs\xi-\bs\eta) \wedge \bs\omega(\bs\eta,t)}{|\bs\xi-\bs\eta|^3} \,,
\end{equation}
where $\bs\omega = \bs \omega(\bs\xi,t) = \nabla \wedge \bs u(\bs\xi,t)$ is the vorticity, $\bs\xi = (\xi_1,\xi_2,\xi_3)$ denotes a point in $\bb R^3$, and $t\in \bb R_+$ is the time. The equations are completed by the initial conditions. It is worthwhile to emphasize that the incompressibility condition $\nabla\cdot \bs u=0$ is clearly verified in view of Eq.~\eqref{u-vort}.

Denoting by $(z,r,\theta)$ the cylindrical coordinates, we recall that the vector field $\bs F$ of cylindrical components $(F_z, F_r, F_\theta)$ is called axisymmetric without swirl if $F_\theta=0$ and $F_z$ and $F_r$ are independent of $\theta$. 

The axisymmetry is preserved by the evolution Eqs.~\eqref{vorteq}, \eqref{u-vort}. Moreover, when restricted to axisymmetric velocity fields $\bs u(\bs\xi,t) = (u_z(z,r,t), u_r(z,r,t), 0)$, the vorticity is
\begin{equation}
\label{omega}
\bs\omega = (0,0,\omega_\theta) = (0,0,\partial_z u_r - \partial_r u_z)
\end{equation}
and, denoting henceforth $\omega_\theta$ by $\omega$, Eq.~\eqref{vorteq} reduces to
\begin{equation}
\label{omeq}
\partial_t \omega + (u_z\partial_z + u_r\partial_r) \omega - \frac{u_r\omega}r = 0 \,.
\end{equation}
Finally, by Eq.~\eqref{u-vort}, $u_z = u_z(z,r,t)$ and $u_r=u_r(z,r,t)$ are given by
\begin{align}
\label{uz}
u_z & = - \frac{1}{2\pi} \int\! \rmd z' \!\int_0^\infty\! r' \rmd r' \! \int_0^\pi\!\rmd \theta \, \frac{\omega(z',r',t) (r\cos\theta - r')}{[(z-z')^2 + (r-r')^2 + 2rr'(1-\cos\theta)]^{3/2}} \,,
\\ \label{ur}
u_r & = \frac{1}{2\pi} \int\! \rmd z' \!\int_0^\infty\! r' \rmd r' \! \int_0^\pi\!\rmd \theta \, \frac{\omega(z',r',t) (z - z')}{[(z-z')^2 + (r-r')^2 + 2rr'(1-\cos\theta)]^{3/2}} \,.
\end{align}
In conclusion, the axisymmetric solutions to the Euler equations are the solutions to Eqs.\ \eqref{omeq}, \eqref{uz}, and \eqref{ur}.

We notice that Eq.~\eqref{omeq} means that the quantity $\omega/r$ remains constant along the flow generated by the velocity field, i.e., 
\begin{equation}
\label{cons-omr}
\frac{\omega(z(t),r(t),t)}{r(t)} = \frac{\omega(z(0),r(0),0)}{r(0)} \,,
\end{equation}
with $(z(t),r(t))$ solution to
\begin{equation}
\label{eqchar}
\dot z(t) = u_z(z(t),r(t),t) \,, \qquad \dot r(t) = u_r(z(t),r(t),t) \,.
\end{equation}

We notice that in the case of non-smooth initial data, Eqs.~\eqref{uz}, \eqref{ur}, \eqref{cons-omr}, and \eqref{eqchar} can be assumed as a weak formulation of the Euler equations in the framework of axisymmetric solutions. An equivalent weak formulation is obtained from Eq.~\eqref{omeq} by a formal integration by parts,
\begin{equation}
\label{weq}
\frac{\rmd}{\rmd t} \omega_t[f] = \omega_t[u_z\partial_z f + u_r\partial_r f + \partial_t f] \,,
\end{equation}
where $f = f(z,r,t)$ is any bounded smooth test function and 
\[
\omega_t[f] := \int\! \rmd z \!\int_0^\infty\! \rmd r \, \omega(z,r,t) f(z,r,t) \,.
\]

It is known that the global (in time) existence and uniqueness of a weak solution to the associate Cauchy problem holds when initial vorticity is a bounded function with compact support contained in the open half-plane $\Pi:=\lbrace(z,r):r>0\rbrace$, see, for instance, \cite[Page 91]{MaP94} or \cite[Appendix]{CS}. In particular, it can be shown that the support of the vorticity remains in the open half-plane $\Pi$ at any time (note that a point in the half-plane $\Pi$ corresponds to a circumference in the three dimensional space $\bb R^3$).

The special class of axisymmetric solutions without swirl are known in literature as
\textit{smoke  rings} (or \textit{vortex rings}),  because of the property to preserve  
their shape in time, which translates with a constant speed along the $z$-axis.
The knowledge of such solutions is very old, but
a first rigorous proof of existence and properties of these solutions (in the stationary case) 
goes back to \cite{FrB74, AmS89},  by means of variational methods. Other information and references on axially symmetric solution without swirl can be found in \cite{ShL92}.

We consider in the present paper  the special class of initial data for which the vorticity is initially very concentrated. We mean that, given a small parameter $\eps\in (0,1)$, we take initial data for which the vorticity has compact support contained in $N$ disks, that is
\begin{equation}
\label{in}
\omega_\eps(z,r,0) = \sum_{i=1}^N  \omega_{i,\eps}(z,r,0) \,,
\end{equation}
where $\omega_{i,\eps}(z,r,0)$, $i=1,\ldots, N$, are functions with definite sign whose support
is contained  in $\Sigma(\xi|\rho)$, which is the open disk of center $\xi$ and radius $\rho$,
\begin{equation}
\label{initial}
\Lambda_{i,\eps}(0) := \supp\, \omega_{i,\eps}(\cdot,0) \subset \Sigma(\zeta^i|\eps),
\end{equation}
with 
\[
\overline{\Sigma(\zeta^i|\eps)} \subset\Pi\quad \forall\, i\,, \qquad  \Sigma(\zeta^i|\eps)\cap \Sigma(\zeta^j|\eps)=\emptyset\quad \forall\, i \ne j,
\]
for fixed $\zeta^i = (z_i,r_i)\in \Pi$. We assume also that
\begin{equation}
\label{2D}
\min_i r_i > 2D\quad \forall\, i\,, \qquad |r_i-r_j| \ge 2D \quad \forall\, i \ne j \,,
\end{equation}
where $D$ is a positive fixed constant. This means that the annuli have different radii. 

In view of \eqref{cons-omr}, the decomposition Eq.~\eqref{in} extends to positive time setting
\begin{equation}
\label{in-t}
\omega_\eps(z,r,t) = \sum_{i=1}^N  \omega_{i,\eps}(z,r,t) \,,
\end{equation}
with $\omega_{i,\eps}(x,t)$ the time evolution of the $i$th vortex ring,
\begin{equation}
\label{cons-omr_ni}
\omega_{i,\eps}(z(t),r(t),t) := \frac{r(t)}{r(0)} \omega_{\eps,i}(z(0),r(0),0)\,.
\end{equation}

We focus on the case of a fluid with a large vorticity concentration. Therefore, in order to have non trivial (i.e., neither vanishing nor diverging) limiting velocities of the vortex rings, the initial data have to be chosen appropriately. The correct choice can be inferred by considering the simplest case of a vortex ring alone, of intensity $N_\eps =: \int\! \rmd z \int _0^\infty\! \rmd r \, \omega_\eps(z,r,0)$ and supported in a small region of diameter $\eps$. It is well known that it moves along the $z$-direction with an approximately constant speed proportional to $N_\eps |\log \eps|$, see \cite {Fr70}. With this in mind, we assume that there are $N$ real parameters $a_1,\ldots, a_N$, called \textit{vortex intensities}, such that
\begin{equation}
\label{ai}
|\log\eps| \int\!\rmd z \!\int_0^\infty\!\rmd r\, \omega_{i,\eps}(z,r,0) = a_i \quad \forall\,i=1,\ldots,N \,.
\end{equation}
Finally, to avoid too large vorticity concentrations, we further assume there is a constant $M>0$ such that 
\begin{equation}
\label{Mgamma}
|\omega_{i,\eps}(z,r,0)| \le \frac{M}{\eps^2|\log\eps|} \quad \forall\, (z,r)\in \Pi \quad \forall\, i=1,\ldots,N\,.
\end{equation}
Now, we can state the main result of the paper.

\begin{theorem}
\label{thm:1}
Assume the initial data $\omega_\eps(x,0)$ verify Eqs.~\eqref{in}, \eqref{initial}, \eqref{2D}, \eqref{ai}, and \eqref{Mgamma}, and define
\begin{equation}
\label{free_m}
\zeta^i(t) :=  \zeta^i + \frac{a_i}{4\pi r_i} \begin{pmatrix} 1 \\ 0 \end{pmatrix} t \,, \quad i=1,\ldots,N\,.
\end{equation}
Then, for any $T>0$ the following holds true. For any $\eps$ small enough there are $\zeta^{i,\eps}(t)\in \Pi$, $t\in [0,T]$, $ i=1,\ldots,N$, and $R_\eps>0$ such that
\[
\lim_{\eps\to 0}|\log\eps| \int_{\Sigma(\zeta^{i,\eps}(t)|R_\eps)}\!\rmd z\,\rmd r\, \omega_{i,\eps}(z,r,t) = a_i\quad \forall\, i=1,\ldots,N, \quad \forall\, t\in [0,T]\,,
\]
with
\[
\lim_{\eps\to 0} R_\eps = 0, \qquad \lim_{\eps\to 0} \zeta^{i,\eps}(t) = \zeta^i(t) \quad \forall\, t\in [0,T]\,.
\]
\end{theorem}

\begin{remark}
\label{rem:1}
For the sake of concreteness, we make the assumption Eq.~\eqref{2D} which guarantees $|\zeta^i(t)-\zeta^j(t)|\ge 2D$ for any $i\ne j$ and $t\ge 0$. On the other hand, as it will be clear from the proof, the result is true for any choice of initial conditions $\{\zeta^i\}$ and intensities $\{a_i\}$ provided that the trajectories $\{\zeta^i(t)\}$ remain separate from each other at any positive time. While, in the general case, the statement of the theorem remains valid only for $T<T_*$, where $T_*$ is the first collapsing time (obviously, the initial data which produce collapses are exceptional).
\end{remark}

\begin{remark}
\label{rem:2}
Sometimes, the Euler equations are considered with initial vorticity highly concentrated around a generic curve, say $\Gamma= \{\bs \gamma_\sigma\}_{\sigma\in [0,1]} \subset \bb R^3$, see, e.g., \cite{MaB02}. Of course, additional assumptions are needed to analyze the time evolution. Here, the main feature is the so-called LIA approximation (local induction approximation), in which the vorticity remains concentrated around a \textit{vortex filament} $\Gamma(t)= \{\bs\gamma_\sigma(t)\}_{\sigma\in [0,1]}$, whose time velocity $\dot{\bs \gamma}_\sigma(t)$ depends on the curvature and is directed along the binormal vector. In the present paper, we present a situation in which this approximation is rigorously derived.
\end{remark}

\begin{remark}
\label{rem:3}
The effect of a viscosity perturbation in the derivation of the vortex model has been discussed in the literature  \cite{CS,Mar90,Mar98,Mar07,Gal11}, but this topic is out of the purposes of the present analysis.
\end{remark}

\begin{remark}
\label{rem:4}
In this paper, we show that for certain classes of concentrated initial data the time evolution is closely related to the dynamics of a particular system of particles. Actually, the relation between the solution of the Euler Equations and time evolution of some special particle systems is more general and it is at the basis of an approximation method, called ``vortex method'', widely used in literature, see, e.g., \cite{CGP14} or the textbook \cite{MaP94}.
\end{remark}

The strategy in the proof of Theorem \ref{thm:1} is the same of the previous works on the topics. (We quote here only the more recent ones \cite{BuM1,BuM2},  and address the reader to the references therein). We first show the corresponding result for a ``reduced system'', where a vortex ring alone moves under the action of a suitable  external time-dependent vector field. The result for the original model is then achieved by treating the motion of each vortex ring as that of a reduced system, in which the external field describes the force due to its interaction with the other rings.

The key tool in the planar case \cite{BuM1} is a sharp a priori estimate on the moment of inertia, which is not available in the axial symmetric case because the velocity field is not a Lipschitz function. To overcome this problem, in \cite{BuM2} the energy conservation is used to control the growth in time of the moment of inertia, which allows us to build up an iterative scheme to deduce the sharp localization property, but the price to pay is that this scheme converges only for short times. 

Theorem \ref{thm:1} extends the result of \cite{BuM2} globally in time, and the strategy behind this improvement relies in the following observation. A suitable decomposition of the velocity field shows that its non Lipschitz part is directed along the $z$-axis, which suggests that the vorticity should stay more localized along the radial direction. Indeed, this is true and allows us to deduce a sharper estimate on a different quantity, the ``axial moment of inertia''. This new estimate makes possible to build up an iterative scheme as in \cite{BuM2}, but here convergent at any positive time, thus deducing a sharp localization property globally in time. 

The plan of the paper is the following. In the next section we introduce the reduced system and prove Theorem \ref{thm:1} as a corollary of the analogous result for this system, which is proved in Sections \ref{sec:3} and \ref{sec:4}. Finally, in Appendix \ref{app:a} we extend to the reduced system a concentration property of the vorticity distribution, proved in \cite{BCM00} for the case of a vortex alone. This property is necessary to characterize the axial motion and its proof relies on an accurate control on the time variation of the energy.

\section{Reduction to a single vortex problem}
\label{sec:2}

We rename the variables by letting
\begin{equation}
\label{nv}
x = (x_1,x_2) := (z,r)
\end{equation}
and extend the vorticity to a function on the whole plane by setting $\omega_\eps(x,t) = 0$ for $x_2\le 0$, so that $x=(x_1,x_2) \in\bb R^2$ henceforth. In this way, the equations of motion Eqs.~\eqref{uz}, \eqref{ur}, \eqref{cons-omr}, and \eqref{eqchar} take the following form,
\begin{equation}
\label{u=}
u(x,t) = \int\!\rmd y\, H(x,y)\, \omega_\eps(y,t)\,,
\end{equation}
\begin{equation}
\label{cons-omr_n}
\omega_\eps(x(t),t) = \frac{x_2(t)}{x_2(0)} \omega_\eps(x(0),0) \,, 
\end{equation}
\begin{equation}
\label{eqchar_n}
\dot x(t) = u(x(t),t) \,,
\end{equation}
where $u(x,t) = (u_1(x,t), u_2(x,t))$ and the kernel $H(x,y) = (H_1(x,y),H_2(x,y))$ is given by
\begin{align}
\label{H1}
H_1(x,y) & = \frac{1}{2\pi} \int_0^\pi\!\rmd \theta \, \frac{y_2(y_2 - x_2\cos\theta)}{\big[|x-y|^2 + 2x_2y_2(1-\cos\theta)\big]^{3/2}} \,,
\\ \label{H2}
H_2(x,y) & = \frac{1}{2\pi} \int_0^\pi\!\rmd \theta \, \frac{y_2 (x_1-y_1) \cos\theta}{\big[|x-y|^2 + 2x_2y_2(1-\cos\theta)\big]^{3/2}} \,.
\end{align}

The ``reduced system'' describes the motion of a single vortex ring in a suitable external time-dependent vector field, which simulates the interaction with the other vortices. This system is defined by Eqs.~\eqref{u=}, \eqref{cons-omr_n}, and, in place of Eq.~\eqref{eqchar_n}, 
\begin{equation}
\label{eqchar_nF}
\dot x(t) = u(x(t),t) + F^\eps(x(t),t)\,.
\end{equation}
The initial datum $\omega_\eps(x,0)$ and the time dependent vector field $F^\eps$  are assumed to satisfy the following conditions. 

\begin{assumption}
\label{ass:1}
The function $\omega_\eps(x,0)$ is non-negative (resp.~non-positive) and there is $M>0$ and $a>0$ (resp.~$a<0$) such that 
\begin{equation}
\label{MgammaF}
0 \le |\omega_\eps(x,0)| \le \frac{M}{\eps^2|\log\eps|} \quad \forall\, x\in\bb R^2\,, \qquad |\log\eps|\int\!\rmd y\, \omega_\eps(y,0) =a\,.
\end{equation}
Moreover, there exists $\zeta^0 = (z_0,r_0)$, with $r_0>0$, such that
\begin{equation}
\label{initialF}
\Lambda_\eps(0) := \supp\, \omega_\eps(\cdot,0) \subset \Sigma(\zeta^0|\eps)\,.
\end{equation}
Finally, $F^\eps=(F^\eps_1,F^\eps_2)$ is a continuous and globally Lipschitz vector field, and it enjoys the following properties.
\begin{itemize}
\item[(a)] $\bs F^\eps = (F^\eps_z,F^\eps_r,F^\eps_\theta) := (F^\eps_1,F^\eps_2,0)$ has zero divergence, i.e., $\partial_{x_1}(x_2 F^\eps_1) + \partial_{x_2}(x_2 F^\eps_2) = 0$.
\item[(b)] There exist $C_F, L >0$ such that, for any $\eps\in (0,1)$ and $t\ge 0$,
\begin{equation}
	\label{2Lipsc}
	|F^\eps(x,t)| \le \frac{C_F}{|\log\eps|}\,, \quad |F^\eps(x,t) - F^\eps(y,t)| \le \frac{L}{|\log\eps|} |x-y|\qquad \forall\,x,y\in\bb R^2\,.
\end{equation}
\end{itemize}
\end{assumption}

\begin{theorem}
\label{thm:2}
Under Assumption \ref{ass:1}, let
\begin{equation}
\label{zett}
\zeta(t) =  \zeta^0+ \frac{a}{4\pi r_0} \begin{pmatrix} 1 \\ 0 \end{pmatrix} t\,.
\end{equation}
Then, for each $T>0$ the following holds true.

\begin{itemize}
\item[(1)] For any $k\in \big(0,\frac 14\big)$ there is $C_k>0$ such that, for any $\eps$ small enough,
\[
\Lambda_\eps(t) := \supp\, \omega_\eps(\cdot,t) \subset \{x\in \bb R^2 \colon |x_2-r_0| \le C_k |\log\eps|^{-k}\} \quad \forall\,t\in [0,T]\,.
\]
\item[(2)] For any $\eps$ small enough there are $\zeta^\eps(t)\in \Pi$, $t\in [0,T]$, and $\varrho_\eps>0$ such that
\[
\lim_{\eps\to 0}|\log\eps| \int_{\Sigma(\zeta^\eps(t)|\varrho_\eps)}\!\rmd x\, \omega_\eps(x,t) = a \,,
\]
with
\[
\lim_{\eps\to 0} \varrho_\eps = 0, \quad \lim_{\eps\to 0} \zeta^\eps(t) = \zeta(t) \,.
\]
\end{itemize}
\end{theorem}

\subsection{Proof of Theorem \ref{thm:1}}
\label{sec:2.3}

Given $T$ as in the statement of the theorem, we fix $R<D$ and let 
\[
T_\eps := \max\left\{t\in [0,T] \colon |x_2 - r_i| \le R \;\; \forall\, x\in \Lambda_{i,\eps}(s) \;\; \forall\, s\in [0,t]\;\;\forall\, i=1,\ldots,N \right\}.
\]
By continuity, from Eq.~\eqref{in} and \eqref{initial} it follows that $T_\eps >0$ for any $\eps$ sufficiently small. Moreover, in view of Eq.~\eqref{2D}, for any $t\in [0,T_\eps ]$ the rings evolve with supports $\Lambda_{i,\eps}(t)$ that remain separated from each other by a distance larger than or equal to $2(D-R)$, and hence their mutual interaction remains bounded and Lipschitz. More precisely, during the time interval $[0,T_\eps]$, the $i$-th vortex ring $\omega_{i,\eps}(x,t)$ evolves according to a reduced system, with external field in Eq.~\eqref{eqchar_nF} given by
\begin{equation}
\label{fk1}
F^{i,\eps}(x,t) = \sum_{j: j\ne i} \int\!\rmd y\, \tilde H(x,y)\, \omega_{j,\eps}(y,t)\;,
\end{equation}
where $\tilde H(x,y)$ is any smooth kernel such that, e.g., $\tilde H (x,y) = H(x,y)$ if $|x-y|\ge D-R$. In view of the explicit form Eqs.~\eqref{H1}, \eqref{H2} of $H$, and the assumption Eq.~\eqref{ai}, $\tilde H$ can be chosen such that $\bs F^{i,\eps} := (F^{i,\eps}_1,F^{i,\eps}_2,0)$ has zero divergence\footnote{This mollification is obtained by modifying the stream function associated to the field, which always exists for axisymmetric flow without swirl \cite[Section 2]{FrB74}.} and, for some constant $\overline C>0$, any $i,j=1,\ldots, N$, and $t\in [0,T_\eps]$,
\[
|F^{i,\eps}(x,t)| \le \frac{\overline C}{|\log\eps|}, \quad |F^{i,\eps}(x,t) - F^{j,\eps}(y,t)| \le \frac{\overline C}{|\log\eps|} |x-y|\quad \forall\,x,y\in\bb R^2.
\]
We then apply Theorem \ref{thm:2} to the evolution of the $i$-th vortex ring, with parameters $(a_i,\zeta^i,T,k)$ in place of $(a,\zeta^0,T,k)$, and conclude that, for any $\eps$ small enough,

\smallskip\noindent
(1) $ |x_2 - r_i| \le C_k|\log\eps|^{-k}$ for any $x\in \Lambda_{i,\eps}(t)$, $t\in[0, T_\eps]$, and $i=1,\ldots,N$,

\smallskip\noindent
(2) there are $\zeta^{i,\eps}(t)\in \Pi$, $i=1,\ldots,N$, and $\varrho_\eps>0$ such that
\[
\lim_{\eps\to 0}|\log\eps| \int_{\Sigma(\zeta^{i,\eps}(t)|\varrho_\eps)}\!\rmd x\, \omega_\eps(x,t) = a_i\,,
\]
with
\[
\lim_{\eps\to 0} \varrho_\eps = 0, \quad \lim_{\eps\to 0} \zeta^{i,\eps}(t) = \zeta^i(t) \,.
\]
By continuity, $T_\eps=T$ for any $\eps$ small enough, and Theorem \ref{thm:1} is thus proved.
\qed

\section{The reduced system: analysis of the radial motion}
\label{sec:3}

The proof Theorem \ref{thm:2} is split in two parts. The first one, which is the content of the present section, concerns the sharp localization property of the vorticity along the radial direction as stated in item (1) of Theorem \ref{thm:2}. Without loss of generality, we consider the case $a=1$, hence Eq.~\eqref{MgammaF} reads
\begin{equation}
\label{MgammaFb}
0 \le \omega_\eps(x,0) \le \frac{M}{\eps^2|\log\eps|} \quad \forall\, x\in\bb R^2, \qquad |\log\eps|\int\!\rmd y\, \omega_\eps(y,0) =1.
\end{equation}

The following weak formulation will be used, which is a direct generalization of Eq.~\eqref{weq},
\begin{equation}
\label{weqF}
\frac{\rmd}{\rmd t} \int\!\rmd x\, \omega_\eps(x,t) f(x,t) = \int\!\rmd x\, \omega_\eps(x,t) \big[ (u+F^\eps) \cdot \nabla f + \partial_t f \big](x,t) \,,
\end{equation}
where $f = f(x,t)$ is any bounded smooth test function. Moreover, the kernel $H(x,y)$ in Eq.~\eqref{u=} can be split as made in \cite[Lemma 3.3]{BuM2}, where it is shown that the most singular part of $H(x,y)$ is given by the kernel $K(x-y)$ corresponding to the planar case, 
\begin{equation}
\label{vel-vor3}
K(x) = \nabla^\perp  G(x)\,, \quad G(x) := - \frac{1}{2\pi} \log|x|\,,
\end{equation}
where $v^\perp :=(v_2,-v_1)$ for $v = (v_1,v_2)$. More precisely, for any $x,y\in \Pi$,
\begin{equation}
\label{sH}
H(x,y) = K(x-y) + L(x,y) + \mc R(x,y)\,,
\end{equation}
where
\begin{equation}
\label{bc_bound}
L(x,y) = \frac{1}{4\pi x_2} \log\frac {1+|x-y|}{|x-y|}\begin{pmatrix} 1 \\ 0 \end{pmatrix}
\end{equation}
and there exists $C_0>0$ such that, for any $x,y\in \Pi$,
\begin{equation}
\label{sR}
|\mc R(x,y)| \le C_0 \frac{1+x_2+\sqrt{x_2y_2} \big(1+ |\log(x_2y_2)|\big)}{x_2^2}\,.
\end{equation}

\noindent
\textit{A notation warning:} In what follows, we shall denote by $C$ a generic positive constant, whose numerical value may change from line to line and it may possibly depend on the parameters $\zeta^0=(z_0,r_0)$ and $M$ appearing in Theorem \ref{thm:2} and Eq.~\eqref{MgammaFb}, as well as on the given time $T$.

\smallskip
As claimed at the beginning of the section, our goal is to show that, under Assumption \ref{ass:1}, for any $T>0$ and $k\in\big(0,\frac 14\big)$, if $\eps$ is small enough then
\begin{equation}
\label{eq:prop1}
|x_2 - r_0| \le \frac{C}{|\log\eps|^k} \qquad \forall\, x\in \Lambda_\eps(t) \quad \forall\, t\in [0,T]\,.
\end{equation}

We let
\[
T^0_\eps := \max\left\{t\in [0,T] \colon \frac{r_0}{2} \leq x_2 \leq \frac32 r_0\;\; \forall\, x \in\Lambda_\eps(s)\;\; \forall\, s\in [0,t] \right\}
\]
and assume hereafter $\eps < r_0/2$ so that $T^0_\eps >0$ in view of Eq.~\eqref{initialF}. In what follows, we show that, for any $k\in\big(0,\frac 14\big)$,
\begin{equation}
\label{stimGa}
|x_2 - r_0| \le \frac{C}{|\log\eps|^k} \qquad \forall\, x\in \Lambda_\eps(t) \quad \forall\, t\in [0,T^0_\eps]\,, 
\end{equation}
provided $\eps$ is small enough. By continuity, this implies that $T^0_\eps = T$ (for $\eps$ sufficiently small), from which Eq.~\eqref{eq:prop1} follows for $\eps$ sufficiently small. 

The proof of Eq.~\eqref{stimGa} is quite long, so it is divided into three preliminary lemmas plus a conclusion. Preliminarily, it is useful to decompose the velocity field according to  Eq.~\eqref{sH}, writing
\begin{equation}
\label{decom_u}
u(x,t) = \widetilde u(x,t)  + \int\!\rmd y\, L(x,y)\, \omega_\eps(y,t) +  \int\!\rmd y\, \mc R(x,y)\, \omega_\eps(y,t) \,,
\end{equation}
where $\widetilde u(x,t)=\int\!\rmd y\, K(x-y)\, \omega_\eps(y,t)$. 

\begin{lemma}
\label{lem:1}
The following estimates hold true,
\begin{equation}
\label{bc_bound2}
\int\!\rmd y\, |L(x,y)|\, \omega_\eps(y,t) \le C \,, \quad \int\!\rmd y\, |\mc R(x,y)| \, \omega_\eps(y,t) \le \frac{C}{|\log\eps|} \qquad \forall\, t\in [0,T^0_\eps]\,.
\end{equation}
\end{lemma}

\begin{proof}
From Eq.~\eqref{bc_bound} and \eqref{sR} it follows that
\begin{equation}
\label{lrest}
|L(x,y)| \le \frac{1}{2\pi r_0}\log\frac {1+|x-y|}{|x-y|}\,,\quad |\mc R(x,y)| \le C \quad \forall\, x,y\in \Lambda_\eps(t)\quad \forall\,  t\in [0,T^0_\eps]\,,
\end{equation}
while, from Eq.~\eqref{cons-omr_n}, \eqref{MgammaFb}, and the definition of $T^0_\eps$, 
\begin{equation}
\label{omega_t}
|\omega_\eps(x,t)| \le \frac{3M}{\eps^2|\log\eps|} \quad \forall\, t\in [0,T^0_\eps]\,.
\end{equation}
Since $\log\frac {1+|x-y|}{|x-y|}$ is monotonically unbounded as $y\to x$, the maximum of the function $\int\!\rmd y\, \log\frac {1+|x-y|}{|x-y|}\, \omega_\eps(y,t)$ is achieved when we rearrange the vorticity mass as close as possible to the singularity. Therefore, in view of Eq.~\eqref{omega_t},
\begin{equation}
\label{311b}
\begin{split}
	\int\!\rmd y\, |L(x,y)|\, \omega_\eps(y,t) & \le \frac{1}{2\pi r_0} \int\!\rmd y\, \log\frac {1+|x-y|}{|x-y|}\, \omega_\eps(y,t) \\ & \le \frac{3M}{\eps^2|\log\eps|r_0} \int_0^{\bar\rho}\!\rmd \rho\, \rho \, \log\frac{1+\rho}{\rho} \\ & =  \frac{3M}{\eps^2|\log\eps|r_0} \bigg\{\frac{\bar\rho^2}{2} \log\frac{1+\bar\rho}{\bar\rho} - \frac 12 \int_0^{\bar\rho}\!\rmd \rho\, \frac{\rho}{1+\rho} \bigg\},
\end{split}
\end{equation}
with $\bar\rho$ such that $3\pi\bar\rho^2 M/(\eps^2|\log\eps|)= 1/|\log\eps|$, from which the first estimate in Eq.~\eqref{bc_bound2} follows. Finally, we observe that, by Liouville's theorem and Eq.~\eqref{cons-omr_n}, since the vector field $\bs F^\eps$ in Assumption \ref{ass:1}-(a)  has zero divergence,
\begin{equation}
\label{w=1}
\int\! \rmd y\, \omega_\eps(y,t) =  \int\! \rmd\bs\xi\, \frac{\omega_\eps(\bs\xi,t)}{r} =  \int\! \rmd\bs\xi_0 \, \frac{\omega_\eps(\bs\xi_0,0)}{r_0} = \int\! \rmd y\, \omega_\eps(y,0)  =\frac{1}{|\log\eps|}\, ,
\end{equation}
where we have used  the coordinate transformation  $\bs\xi=\phi^t(\bs\xi_0)$, with $\phi^t$ the flow generated by $\dot{\bs\xi} = \bs u(\bs\xi,t) + \bs F^\eps(\bs\xi,t)$. Therefore, the second estimate in Eq.~\eqref{bc_bound2} is a consequence of the second one in Eq.~\eqref{lrest}.
\end{proof}

We denote by $B_\eps(t)=(B_{\eps,1}(t), B_{\eps,2}(t))$ the center of vorticity of the blob, defined by
\begin{equation}
\label{c.m.}
B_\eps(t) = \frac{\int\! \rmd x\, x\, \omega_\eps(x,t)}{\int\! \rmd x\, \omega_\eps(x,t)}
=|\log\eps|\int\! \rmd x\, x\, \omega_\eps(x,t) \,, 
\end{equation}
and by $I_\eps(t)$ the axial moment of inertia with respect to $x_2=B_{\eps,2}(t)$, i.e.,
\begin{equation}
\label{moment}
I_\eps(t) = \int\! \rmd x\, \left(x_2-B_{\eps,2}(t)\right)^2  \omega_\eps(x,t)\,.
\end{equation}
Since $\Lambda_\eps(t)$ is compact, the time derivatives of $B_{\eps,2}(t)$ (in this section, we are only interested in this component) and $I_\eps(t)$  can be computed by means of Eq.~\eqref{weqF}. To this end, we first observe that the time derivative of $M_2 := \int\!\rmd x\, \omega(x,t) x_2^2$ is $\dot M_2 = \int\! \rmd x \, \omega_\eps(x,t) \,2 x_2  F^\eps_2(x,t)$ (it is a conserved quantity in absence of external field, see Appendix \ref{app:a}), so that
\begin{equation}
\label{growth B}
\dot B_{\eps,2}(t) =  |\log\eps|\int\! \rmd x\,\omega_\eps(x,t)\, \left(F^\eps_2(x,t) +\int\!\rmd y\, \mc R_2(x,y)\, \omega_\eps(y,t)  \right),
\end{equation}
\begin{equation}
\begin{split}
\label{growth moment}
\dot I_\eps(t) & = 2 \int\! \rmd x\,\omega_\eps(x,t)\, (x_2 - B_{\eps,2}(t)) F^\eps_2(x,t)  \\ & \quad - 2 B_{\eps,2}(t) \int\!\rmd x\, \omega_\eps(x,t) \int\!\rmd y\, \mc R_2(x,y)\, \omega_\eps(y,t) \,,
\end{split}
\end{equation}
where we have used the expression Eq.~\eqref{decom_u} for $u(x,t)$, the identities 
\[
\int\! \rmd x\,\omega_\eps(x,t) \,(x_2-B_{\eps,2}(t)) = 0\,, \quad \int\!\rmd x\, \widetilde u(x,t)\,\omega_\eps(x,t) = 0\,,
\]
(which derive from the definition of center of vorticity and the explicit form of $K(x-y)$ in Eq.~\eqref{vel-vor3}), and the fact that $L_2(x,y)=0$, see Eq.~\eqref{bc_bound}.

\begin{lemma}
\label{lem:2}
The following estimate holds, 
\begin{equation}
\label{Iee}
I_\eps(t) \le \frac{C}{|\log\eps|^2} \qquad \forall\, t\in [0,T^0_\eps]\,.
\end{equation}
\end{lemma}

\begin{proof}
By Eqs.~\eqref{2Lipsc},  \eqref{bc_bound2}, \eqref{growth B}, and \eqref{growth moment} we have that, for any $[0,T^0_\eps]$, $|\dot B_{\eps,2}(t)| \le C/|\log\eps|$ (hence $|B_{\eps,2}(t)| \le C$) and  
\[
|\dot I_\eps(t)| \le \frac{C}{|\log\eps|} \int\! \rmd x\,  |x_2-B_{\eps,2}(t)| \, \omega_\eps(x,t) + \frac{C}{|\log\eps|^2} \le \frac{C}{|\log\eps|^{3/2}}\sqrt{I_\eps(t)} + \frac{C}{|\log\eps|^2}\,,
\]
where in the last estimate we used the Cauchy-Schwarz inequality and  Eq.~\eqref{w=1}. Eq.~\eqref{Iee} now follows by integration of the last differential inequality since the initial data imply $I_\eps(0) \leq 4\eps^2$.
\end{proof}

\begin{lemma}
\label{lem:3}
Recall $\Lambda_\eps(t)=\supp\omega_\eps(\cdot,t)$ and define
\begin{equation}
\label{Rt}
R_t:= \max\{|x_2-B_{\eps,2}(t)|\colon x\in \Lambda_\eps(t)\}\,.
\end{equation}
Given $x_0\in\Lambda_\eps(0)$, let $x(x_0,t)$ be the solution to Eq.~\eqref{eqchar_nF} with initial condition $x(x_0,0) = x_0$ and suppose at time $t\in (0,T^0_\eps]$ it happens that  
\begin{equation}
\label{hstimv}
|x_2(x_0,t)-B_{\eps,2}(t)| = R_t\,.
\end{equation}
Then, at this time $t$, 
\begin{equation}
\label{stimv}
\frac{\rmd}{\rmd t} |x_2(x_0,t)- B_{\eps,2}(t)| \leq  \frac{C}{|\log\eps|}+\frac{1}{\pi R_t   |\log\eps|} + \sqrt{\frac{C m_t(R_t/2)}{ \eps^2 |\log\eps|}}\,,
\end{equation}
where the function $m_t(\cdot)$ is defined by 
\begin{equation}
\label{mt}
m_t(h) = \int_{|y_2-B_{\eps,2}(t)|>h}\!\rmd y\,\omega_\eps(y,t)\,.
\end{equation}
\end{lemma}

\begin{proof}
We observe that the proof is similar to that given in \cite[Lemma 2.5]{BuM1}. Letting $x=x(x_0,t)$, by Eqs.~\eqref{eqchar_nF},  \eqref{decom_u},   \eqref{w=1},
and \eqref{growth B} we have,
\begin{equation}
\label{distance1}
\begin{split}
& \frac{\rmd}{\rmd t} |x_2(x_0,t)- B_{\eps,2}(t)| = \big(u_2(x,t) + F^\eps_2(x,t) - \dot B_{\eps,2}(t)\big)  \frac{x_2-B_{\eps,2}(t)}{|x_2-B_{\eps,2}(t)|} \\ & \qquad\qquad = V(x,t) \frac{x_2-B_{\eps,2}(t)}{|x_2-B_{\eps,2}(t)|} + \int\!\rmd y\, K_2(x-y)\, \omega_\eps(y,t) \frac{x_2-B_{\eps,2}(t)}{|x_2-B_{\eps,2}(t)|}\,,
\end{split}
\end{equation}
with
\begin{equation}
\label{j}
\begin{split}
V(x,t) & = F^\eps_2(x,t)  +\int\!\rmd z\, \mc R_2(x,z)\, \omega_\eps(z,t) \\ & \quad - |\log\eps| \int\! \rmd y\,\omega_\eps(y,t)\, \left(F^\eps_2(y,t)+\int\!\rmd z\, \mc R_2(y,z)\, \omega_\eps(z,t) \right).
\end{split}
\end{equation}
From Eqs.~\eqref{2Lipsc}, \eqref{bc_bound2},  and  \eqref{w=1} we have
\begin{equation}
\label{distance4}
|V(x,t)| \le  \frac{C}{|\log\eps|} \quad \forall\, t\in [0,T^0_\eps]\,.
\end{equation}
For the last term in Eq.~\eqref{distance1}, we split the integration region into two parts, the set $A_1= \{y\in \Lambda_\eps(t) \colon |y_2 - B_{\eps,2}(t)|\le R_t/2\}$ and the set $A_2 = \{y\in \Lambda_\eps(t) \colon  R_t/2 < |y_2 - B_{\eps,2}(t)| \le R_t\}$. Then,
\begin{equation}
\label{in A_1,A_2}
\int\!\rmd y\, K_2(x-y)\, \omega_\eps(y,t) \frac{x_2-B_{\eps,2}(t)}{|x_2-B_{\eps,2}(t)|} = H_1 + H_2\,,
\end{equation}
where
\begin{equation}
\label{in A_1}
H_1 = \frac{x_2-B_{\eps,2}(t)}{|x_2-B_{\eps,2}(t)|}  \int_{A_1}\! \rmd y\, K_2(x-y)\, \omega_\eps(y,t) 
\end{equation}
and
\begin{equation}
\label{in A_2}
H_2 = \frac{x_2-B_{\eps,2}(t)}{|x_2-B_{\eps,2}(t)|}  \int_{A_2}\! \rmd y\, K_2(x-y)\, \omega_\eps(y,t)\,.
\end{equation}

We consider first the contribution due to the set $A_1$. Recalling Eq.~\eqref{vel-vor3}, after introducing the new variables $x'=x-B_\eps(t)$,  $y'=y-B_\eps(t)$, we get,
\begin{equation}
\label{in H_11}
|H_1| \leq \frac{1}{2\pi} \int_{|y_2'|\leq R_t/2}\! \rmd y'\, \frac{1}{|x'-y'|}\, \omega_\eps(y'+B_\eps(t))\,.
\end{equation}
From Eq.~\eqref{hstimv} we have $|x'_2| = R_t$, and hence $|y_2'| \le R_t/2$ implies $|x'-y'|\ge |x_2'-y_2'|\geq R_t/2$,
so that
\begin{equation}
|H_1| \leq \frac{1}{\pi \,R_t}  \int_{|y_2'|\leq R_t/2}\! \rmd y'\,  \omega_\eps(y'+B_\eps(t)) 
\leq \frac{1}{\pi R_t   |\log\eps|}\,.
\label{H_14}
\end{equation}

We bound now $H_2$. Again by Eq.~\eqref{vel-vor3},
\begin{equation*}
|H_2| \le \frac{1}{2\pi} \int_{A_2}\! \rmd y\, \frac 1{|x-y|} \, \omega_\eps(y,t)\,.
\end{equation*}
The function $|x-y|^{-1}$ diverges monotonically as $y\to x$, and so the maximum of the integral is obtained when we rearrange the vorticity mass as close as possible to the singularity. By Eq.~\eqref{omega_t} and since, by Eq.~\eqref{mt}, $m_t(R_t/2)$ is equal  to the total amount of vorticity in $A_2$, this rearrangement gives,
\begin{equation}
\label{h2}
|H_2| \le \frac{3M\eps^{-2}}{2\pi |\log\eps|} \int_{\Sigma (0|r)}\!\rmd y'\, \frac{1}{|y'|} = \frac{3M\eps^{-2}}{ |\log\eps|} r \,, 
\end{equation}
where the radius $r$ is such that $3\pi r^2 M/(\eps^2|\log\eps|) = m_t(R_t/2)$. The estimate Eq.~\eqref{stimv} now follows by Eqs.~\eqref{distance1}, \eqref{distance4}, \eqref{in A_1,A_2}, \eqref{H_14}, and \eqref{h2}.
\end{proof}

We determine now the behavior of the function $m_t (\cdot)$ introduced in Eq.~\eqref{mt}
when its argument goes to $0$.  The proof will be adapted from that of  \cite[Proposition 3.4]{BuM2}.

\begin{lemma} 
\label{lem:mt}
Let $m_t$ be defined as in Eq.~\eqref{mt}. For each $\ell>0$  and  $k \in \big(0, \frac 14\big)$,
\begin{equation}
\label{smt}
\lim_{\eps\to 0} \, \max_{ t \in [0, T^0_\eps]} \eps^{-\ell} m_t \left(\frac{1}{|\log\eps|^k} \right)  = 0\,.
\end{equation}
\end{lemma}

\begin{proof}
Given $R\ge 2h^\alpha$, $h>0$,  and  
\begin{equation}
\label{alpha_delta}
\alpha=\frac{1-k}{1+k}-\delta \, , \qquad \delta\in \left(0, \frac{1-2k}{1+k}\right) \, ,
\end{equation}
let $W_{R,h}(x_2)$, with $x_2$ the second component of $x = (x_1, x_2)$, be a non-negative smooth function, such that
\begin{equation}
\label{W1}
W_{R,h}(x_2) = \begin{cases} 1 & \text{if $|x_2|\le R$}, \\ 0 & \text{if $|x_2|\ge R+h$}, \end{cases}
\end{equation}
and its derivative $W_{R,h}'$ satisfies
\begin{equation}
\label{W2}
| W_{R,h}'(x_2)| < \frac{C}{h}\,,
\end{equation}
\begin{equation}
\label{W3}
|W_{R,h}'(x_2)-W_{R,h}'(y_2)| < \frac{C}{h^2}\,|x_2-y_2|\leq \frac{C}{h^2}\,|x - y|\,.
\end{equation}
We introduce the quantity
\begin{equation}
\label{mass 1}
\mu_t(R,h) = \int\! \rmd x \, \big[1-W_{R,h}(x_2-B_{\eps,2}(t))\big]\, \omega_\varepsilon (x,t)\,,
\end{equation}
which is a mollified version of $m_t$ satisfying
\begin{equation}
\label{2mass 3}
\mu_t(R,h) \le m_t(R) \le \mu_t(R-h,h)\,.
\end{equation}
Hence it is sufficient to prove \eqref{smt} with $\mu_t$ in place  of $m_t$. Since  the function $t\mapsto \mu_t(R,h)$ is differentiable, we can compute its time derivative, by  Eq.~\eqref{weqF} with test function $f(x,t) =1- W_{R,h}(x_2-B_{\eps,2}(t))$ and then using Eqs.~\eqref{decom_u} and \eqref{growth B}. We have,
\begin{equation}
\label{mu_t}
\begin{aligned} \frac{\mathrm{d}}{\mathrm{d}t} \mu_t(R,h) = & -
	\int\! \rmd x \, \nabla W_{R,h}(x_2-B_{\varepsilon,2}(t)) \cdot [ u(x,t)+F^\eps(x,t)-\dot{B}_\eps(t) ] \, \omega_\varepsilon(x,t) \\ =  & -
	\int \rmd x \:  W_{R,h}'(x_2-B_{\varepsilon,2}(t))   [u_2(x,t)+F^\eps_2(x,t)-\dot{B}_{\eps,2}(t) ] \, \omega_\varepsilon(x,t) \\ = & -H_3- H_4 \,,
\end{aligned}
\end{equation}
with
\begin{equation*}
\begin{split}
H_3 & = \int\! \rmd x\,   W_{R,h}'(x_2-B_{\eps,2}(t))  \int\!\rmd y \, K_2(x-y)\, \omega_\eps(y,t)\, \omega_\eps(x,t) \\ & = \frac 12 \int\! \rmd x \! \int\! \rmd y\, \omega_\eps(x,t)\,  \omega_\eps(y,t) \\ & \quad \times \left[ W_{R,h}'(x_2-B_{\eps,2}(t)) -   W_{R,h}'(y_2-B_{\eps,2}(t))\right]   K_2(x-y) \,, \\ H_4 & = |\log\eps|\int\! \rmd x\, W_{R,h}'(x_2-B_{\eps,2}(t))\, \omega_\eps(x,t) V(x,t)\,,
\end{split}
\end{equation*}
where the antisymmetry of $K$ has  allowed to achieve the second expression of $H_3$, and $V(x,t)$ is defined in Eq.~\eqref{j}. We can note that, in view of Eq.~\eqref{distance4}, \eqref{W2}, and the fact that $W_{R,h}'(z)$ is zero if $|z|\leq R$,
\begin{equation}
\label{acca4}
H_4 \le \frac{C}{h |\log\eps|} m_t(R)  \quad \forall\, t\in [0,T^0_\eps]\,.
\end{equation}
Now we treat $H_3$. We introduce the new variables $x'=x-B_\eps(t)$, $y'=y-B_\eps(t)$ (as done previously), define $\widetilde\omega_\eps(z,t) := \omega_\eps(z+B_\eps(t),t)$, and 
\[
f(x',y') := \frac 12 \widetilde\omega_\eps(x',t)\, \widetilde\omega_\eps(y',t) \, [W_{R,h}'(x_2')-  W_{R,h}'(y_2')]  K_2(x'-y') \,,
\]
whence  $H_3 = \int\! \rmd x' \! \int\! \rmd y'\, f(x',y')$. We note that $f(x',y')$ is a symmetric function of $x'$ and $y'$ and that, by Eq.~\eqref{W1}, in order to be different from zero it is necessary that either $|x_2'|\ge R$ or $|y_2'|\ge R$. Therefore,
\[
\begin{split}
H_3  &= \bigg[ \int_{|x_2'| > R}\! \rmd x' \! \int\! \rmd y' + \int\! \rmd x' \! \int_{|y_2'| > R}\! \rmd y' -  \int_{|x_2'| > h}\! \rmd x' \! \int_{|y_2'| > R}\! \rmd y'\bigg]f(x',y') \\ & = 2 \int_{|x_2'| > R}\! \rmd x' \! \int\! \rmd y_2'\,f(x',y')  -  \int_{|x_2'| > R}\! \rmd x' \! \int_{|y_2'| > R}\! \rmd y'\,f(x',y') \\ & = H_3' + H_3'' + H_3'''\,,
\end{split}
\]
with 
\[
\begin{split}
H_3' & = 2 \int_{|x_2'| > R}\! \rmd x' \! \int_{|y_2'| \le R-h^\alpha}\! \rmd y'\,f(x',y') \,, \\ H_3''&  = 2 \int_{|x_2'| > R}\! \rmd x' \! \int_{|y_2'| > R-h^\alpha}\! \rmd y'\,f(x',y')\,, \\ H_3''' & = -  \int_{|x_2'| > R}\! \rmd x' \! \int_{|y_2'| > R}\! \rmd y'\,f(x',y')\,.
\end{split}
\]
By the properties of $W_{R,h}$, we have  $W_{R,h}'(y_2') =0$ for $|y_2'| \le R$. In particular,  $W_{R,h}'(y_2') = 0$ for $|y_2'| \le R-h^\alpha$, then
\[
H_3' =  \int_{|x_2'| > R}\! \rmd x' \, \widetilde\omega_\eps(x',t)  W_{R,h}'(x_2')  \int_{|y_2'| \le R-h^\alpha}\! \rmd y'\, K_2(x'-y') \, \widetilde\omega_\eps(y',t)
\]
and therefore, in view of Eq.~\eqref{W2},
\begin{equation}
\label{a1'}
|H_3'| \le \frac{C}{h} m_t(R) \sup_{|x_2'| > R} |A_3(x',t)|\,,
\end{equation}
with
\[
A_3(x',t) =  \int_{|y_2'| \le R-h^\alpha}\! \rmd y'\, K_2(x'-y') \, \widetilde\omega_\eps(y',t) \,.
\]
We note that if $|x_2'| > R$ then $|y_2'| \le R-h^\alpha$ implies $|x'-y'|\ge |x_2'-y_2'|\ge h^\alpha$, hence
\[
\begin{split}
|A_3(x',t)| & \le \frac{1}{2\pi} \int_{|y_2'|\leq R-h^\alpha}\! \rmd y'\, \frac{\widetilde\omega_\eps(y',t) }{|x'-y'|} \\ & \le \frac{1}{2  \pi  h^\alpha} \int_{|y_2'|\le  R-h^\alpha} \! \rmd y'\, \widetilde\omega_\eps(y',t) \le \frac{1}{2\pi h^\alpha |\log\eps|}\,.
\end{split}
\]
We then obtain, by Eq.~\eqref{a1'},
\begin{equation}
\label{H_14b}
|H_3'| \le \frac{C}{h^{1+\alpha} |\log\eps|} m_t(R)\,.
\end{equation}
From Eq.~\eqref{W3}, using Chebyshev's inequality and $R\geq 2 h^\alpha$,
\[
|H_3''| + |H_3'''| \le \frac{C}{h^2} \int_{|x_2'| \ge R}\! \rmd x' \! \int_{|y_2'| \ge R-h^\alpha}\! \rmd y'\,\widetilde\omega_\eps(y',t) \,  \widetilde\omega_\eps(x',t)  \le \frac{C I_\eps(t)}{ h^2 R^2}m_t(R)  \,.
\]
Finally, by Eq.~\eqref{Iee},
\begin{equation}
\label{a1s}
|H_3| \le  C \left( \frac{1}{h^{1+\alpha}|\log\eps|} + \frac{1}{h^2 R^2 |\log\eps|^2} \right) m_t(R)  \quad \forall\, t\in [0,T^0_\eps]\,.
\end{equation}

From estimates Eqs.~\eqref{a1s}  and \eqref{acca4}, recalling Eq.~\eqref{mu_t},  we get,
\begin{equation}
\label{equ_mm}
\frac{\mathrm{d}}{\mathrm{d}t} \mu_t (R,h) \leq A_\varepsilon(R, h) m_t(R) \quad \forall\, t\in [0,T^0_\eps]\,.
\end{equation}
where
\[
A_\varepsilon(R, h) = C \left(\frac{1}{h^{1+\alpha} |\log\eps|} + \frac{1}{h^2 R^2 |\log\eps|^2} +\frac{1}{h |\log\eps|}\right).
\]
Therefore, by Eqs.~\eqref{2mass 3} and \eqref{equ_mm},
\begin{equation}
\mu_t (R,h) \le \mu_0 (R,h) + A_\varepsilon(R, h) \int_0^t {\mathrm{d}} s  \, \mu_s (R-h, h) \quad \forall\, t\in [0,T^0_\eps]\,.
\end{equation}
We assume now $\eps$ sufficiently small, and we iterate the last inequality $n=\lfloor|\log\varepsilon|\rfloor$  times (denoting with $\lfloor a\rfloor$ the integer part of $a>0$), from
\[
R_0 =\frac{1}{|\log\varepsilon|^{k}}  \quad\text{to}\quad R_n=\frac{1}{2|\log\varepsilon|^{k}}\,,
\]
where $R_n = R_0 -n h$,  and consequently
\[
h= \frac{1}{2  n |\log\varepsilon|^{k}} \,.
\]
This procedure is correct because in this range for $R$ the assumption $R\ge 2 h^\alpha$ (under which Eq.~\eqref{equ_mm} has been deduced) is satisfied. Indeed, in view of Eq.~\eqref{alpha_delta} we have
\[
h^\alpha \approx C \left(  \frac{1}{|\log\eps|^{1+k}} \right)^\alpha = C \left(  \frac{1}{|\log\eps|^{1+k}} \right)^{\frac{1-k}{1+k}-\delta} = \frac{C}{|\log\eps|^{1-k-(1+k)\delta}}\,,
\]
with $k<1-k-(1+k)\delta$, and therefore, if $\eps$ is small enough, $h^\alpha \ll R_n$. Moreover, the quantity $A_\varepsilon(R,h)$ is bounded by $C |\log\varepsilon |^q$ with $q<1$, in fact
\[
\begin{split}
\frac{1}{h^{1+\alpha} |\log\eps|} & \le C \frac{\left( |\log\eps|^{k+1}   \right)^{1+\alpha}}{|\log\varepsilon|} \le C |\log\varepsilon |^{1-\delta(k+1)} \,, \\ \frac{1}{h^2 R^2 |\log\eps|^2} & \le C \frac{|\log\eps|^{4k+2} }{|\log\eps|^2} \leq |\log\eps|^{4k}\,, \\ \frac{1}{h |\log\eps|} & \le C |\log\eps|^{k} \,.
\end{split}
\]
In conclusion,
\[
\begin{split}
\mu_t(R_0-h,h) & \le \mu_0(R_0-h,h) + \sum_{j=1}^{n-1} \mu_0(R_j,h) \frac{(C |\log\varepsilon|^q t)^j}{j!} \\ & \quad + \frac{(C |\log\varepsilon|^q )^{n}}{(n-1)!} \int_0^t\!{\textnormal{d}} s\,  (t-s)^{n-1}\mu_s(R_{n},h) \quad \forall\, t\in [0,T^0_\eps]\,.
\end{split}
\]
Since $\Lambda_\varepsilon(0) \subset \Sigma(z|\varepsilon)$, we can determine $\varepsilon$
small enough so that $\mu_0(R_j,h)=0$ for any $j=0,\ldots,n$, hence, for any $t\in [0, T]$,
\begin{equation}
\label{mass 15'}
\mu_t(R_0-h,h) \le \frac{(C |\log\varepsilon|^q)^{n}}{(n-1)!} \int_0^t\!{\textnormal{d}} s\,  (t-s)^{n-1}\mu_s(R_{n},h) \le  \frac{(C |\log\varepsilon|^q  t)^{n}}{n!}\,,
\end{equation}
where in the last inequality we have used the trivial bound $\mu_s(R_{n},h) \le 1$. 
Therefore, using also Eq.~\eqref{2mass 3},  Stirling formula, and $n=\lfloor|\log\varepsilon|\rfloor$,
\[
m_t(R_0) \le \mu_t(R_0 -h,h) \le  \frac{C}{|\log\eps|^{(1-q)|\log\eps|}} \quad \forall\, t\in [0,T^0_\eps]\,,
\]
which implies Eq.~\eqref{smt}.
\end{proof}

\begin{remark}
\label{rem:rc}
In \cite[Prop.~3.4]{BuM2} a similar concentration result is deduced for the vorticity mass outside a small disk, but only for small times. This is due to the non Lipschitz term Eq.~\eqref{bc_bound} (not present in our case), which leads to an estimate like Eq.~\eqref{mass 15'} but with $q=1$. We also remark that a weaker estimate $C/|\log\eps|^\theta$ with $\theta\in (1,2)$ for the axial moment of inertia instead of Eq.~\eqref{Iee} would lead as well to  Eq.~\eqref{mass 15'} (choosing $k\in(0,\frac{\theta-1}4)$ in this case).
\end{remark}

\begin{proof}[Proof of Eq.~\eqref{stimGa}]
In view of Eqs.~\eqref{2Lipsc}, \eqref{bc_bound2}, and \eqref{w=1}, from \eqref{growth B}  and since $|B_{\eps,2}(0)-r_0|\le \eps$ we have,
\begin{equation}
\label{compact2}
|B_{\eps,2}(t) - r_0| \le \frac{C}{|\log\eps|} \quad \forall\, t\in [0,T^0_\eps]\,.
\end{equation}
Therefore, it is sufficient to show that, given $k\in \big(0,\frac 14\big)$,
\begin{equation}
\label{lamb}
|x_2 - B_{\eps,2}(t)| \le \frac{C}{|\log\eps|^k} \quad \forall\, x\in \Lambda_\eps(t) \quad \forall\, t\in [0,T^0_\eps]\,,
\end{equation}
provided $\eps$ is small enough. To this end, we first notice that, in view of Lemma \ref{lem:3} and Eq.~\eqref{Rt}, for any $x_0\in \Lambda_\eps(0)$ and $t\in [0,T^0_\eps]$ we have $|x_2(x_0,t) -B_{\eps,2}(t)| \le R_t$, and whenever $|x_2(x_0,t)- B_{\eps,2}(t)|=R_t$ the differential inequality Eq.~\eqref{stimv} holds true. We claim that this implies
\begin{equation} 
\label{suppRt} 
\Lambda_{\eps}(t)  \subset \{x\in \mathbb{R}^2\colon |x_2-B_{\eps,2}(t)|< \rho(t) \} \quad \forall\, t\in [s_0,s_1] \quad \forall\, [s_0,s_1]  \subseteq [0,T^0_\eps]\,,
\end{equation}
provided $\rho(t)$ solves
\begin{equation} 
\label{eqdiffRt} 
\dot{\rho}(t) = \frac{2C}{|\log\eps|}+\frac{2}{\pi |\log\eps|\rho(t)} + g(t)\,,
\end{equation}
with initial datum $\rho(s_0) > R_{s_0}$ and $g(t)$ any smooth function which is an upper bound for the last term in Eq.~\eqref{stimv}. Indeed, $|x_2-B_{\eps,2}(s_0)| < \rho(s_0)$ for any $x\in \Lambda_\eps(s_0)$ and, by absurd, if there were a first time $t_*\in (s_0,s_1]$ such that $|x_2(x_0,t_*)- B_{\eps,2}(t_*)| = \rho(t_*)$ for some $x_0\in \Lambda_\eps(0)$, it would be necessarily $\rho(t_*) = R_{t_*}$; hence, by Eq.~\eqref{stimv}, $\dot \rho (t_*)$ would be strictly larger than $\frac{\rmd}{\rmd t}|x_2(x_0,t)- B_{\eps,2}(t)|\big|_{t=t_*}$, which contradicts the characterization of $t_*$ as the first time at which the graph of $t\mapsto |x_2(x_0,t)- B_{\eps,2}(t)|$ crosses the one of $t\mapsto \rho(t)$.

Now, let
\[
t_0 = \sup\{t\in [0,T^0_\eps] \colon R_s < 3|\log\eps|^{-k} \;\; \forall\, s\in [0,t]\}\,.
\]
If $t_0 = T^0_\eps$ then Eq.~\eqref{lamb} is already achieved, otherwise we set
\[
t_1 = \sup\{t\in [t_0,T^0_\eps] \colon R_s > 2|\log\eps|^{-k} \;\; \forall\, s\in [t_0,t_1]\}
\]
and consider $\rho(t)$ as in Eq.~\eqref{eqdiffRt}, relative to the interval $[s_0,s_1] = [t_0,t_1]$ and such that
\[
\rho(t_0) = 4|\log\eps|^{-k}\,, \quad g(t) \le \frac{C\eps^{(\ell -2)/2}}{|\log\eps|^{1/2}} \qquad \forall\, t\in [t_0,t_1]\,,
\]
for a fixed $\ell >2$. We note that $\rho(t_0) > R_{t_0}=3|\log\eps|^{-k}$ and that, since $R_t\ge 2|\log\eps|^{-k}$ for any $t\in [t_0,t_1]$, Eq.~\eqref{smt} guarantees that above condition on $g(t)$ is compatible with the requirement that the latter is an upper bound for the last term in Eq.~\eqref{stimv}.

Now, as $\rho(t) \ge R_t \ge 2|\log\eps|^{-k}$ for any $t\in [t_0,t_1]$, the second term in the right-hand side of Eq.~\eqref{eqdiffRt} is bounded by $C|\log\eps|^{k-1}$, and therefore, since $k\in \big(0,\frac 14 \big)$, from Eq.~\eqref{eqdiffRt} we deduce that
\[
\dot \rho(t) \leq \frac{C}{|\log\eps|^{3/4}} \qquad \forall\, t\in [t_0,t_1]\,,
\]
which integrated from $t_0$ and $t_1$ gives
\[
\rho(t) \le \rho(t_0) + \frac{CT }{|\log\eps|^{3/4}} \le \frac{C}{|\log\eps|^k} \qquad \forall\, t\in [t_0,t_1]\,.
\]
Clearly, if $t_1=T^0_\eps$ we are done. Otherwise, we can repeat the same argument in the intervals $[t_0',t_1'] \subseteq [t_1,T^0_\eps]$ defined analogously to $[t_0,t_1]$ (if any). Eq.~\eqref{lamb} is thus proved.
\end{proof}

\section{The reduced system: analysis of the axial motion}
\label{sec:4}

In this section we prove item (2) of Theorem \ref{thm:2}, remarking that, as in the previous section, we always assume $a=1$. We premise a concentration result which shows that  large part of the vorticity remains confined in a disk whose size is infinitesimal as $\eps\to 0$.

\begin{lemma}
\label{lem:4}
Consider the reduced system defined by Eqs.~\eqref{u=}, \eqref{cons-omr_n}, and \eqref{eqchar_nF}. Under Assumption \ref{ass:1}, with Eq.~\eqref{MgammaFb} in place of Eq.~\eqref{MgammaF}, for each $T>0$ there are $\eps_1\in (0,1)$, $C_1>0$ and $q_\eps(t)\in \bb R^2$, such that
\begin{equation}
\label{lem1b}
|\log\eps| \int_{\Sigma(q_\eps(t), \eps|\log\eps|)}\!\rmd x\, \omega_\eps(x,t) \ge 1 - \frac{C_1}{\log|\log\eps|} \qquad \forall\, t\in [0,T] \quad  \forall\, \eps\in (0,\eps_1]\,.
\end{equation}
\end{lemma}

This is the content of \cite[Lemma 3.1]{BuM2} and it is an extension of the analogous result in \cite{BCM00}, where the case without external field is considered. However, since the demonstration given in \cite{BuM2} is affected by an error, we provide here the correct proof, see Appendix \ref{app:a}.

The following proposition shows that in the limit $\eps\to 0$ the center of vorticity performs a motion with constant speed along the $x_1=z$ axis.

\begin{proposition}
\label{prop:2}
Under Assumption \ref{ass:1}, for any $T>0$,
\begin{equation}
\label{bz}
\lim_{\eps\to 0} \max_{t\in [0, T]} |B_\eps(t) - \zeta(t)| = 0\,,
\end{equation}
with $\zeta(t)$ as in Eq.~\eqref{zett}.
\end{proposition}

\begin{proof}
Since $\Lambda_\eps(t)$ is compact we can use Eq.~\eqref{weqF} as in getting Eq.~\eqref{growth B} and compute,
\begin{equation}
\label{bpunto}
\begin{split}
	\dot B_\eps(t) & = |\log\eps|\frac{\rmd}{\rmd t} \int\!\rmd x\, x\,\omega_\eps(x,t) =  |\log\eps| \int\!\rmd x\, \omega_\eps(x,t) \, (u+F^\eps)(x,t) \\ & =  |\log\eps|\int\!\rmd x\, \omega_\eps(x,t) \, F^\eps (x,t) \\ & \quad +  |\log\eps|\int\!\rmd x\, \omega_\eps(x,t) \int\!\rmd y\,  [L(x,y) + \mc R(x,y)]\omega_\eps(y,t)\,,
\end{split}
\end{equation}
where we used Eq.~\eqref{decom_u} and $\int\!\rmd x\, \omega_\eps(x,t) \widetilde u(x,t) =0$.
In what follows, we fix $T>0$ and $k\in \big(0,\frac 12\big)$, and assume the parameter $\eps$ so small in order that Eq.~\eqref{lem1b} does hold and that Eq.~\eqref{stimGa} implies $T^0_\eps=T$. Therefore, from Eq.~\eqref{bpunto},  \eqref{2Lipsc}, in view of Eqs.~\eqref{bc_bound}, \eqref{bc_bound2}, and \eqref{lrest}, we have,
\begin{equation}
\label{b12}
|\dot B_{\eps,1}(t) - Q_\eps(t)| + |\dot B_{\eps,2}(t)| \le  \frac{C}{|\log\eps|}\quad \forall\, t\in [0, T] \,,
\end{equation}
where
\[
Q_\eps(t) := |\log\eps|\int\!\rmd x\, \omega_\eps(x,t) \frac{1}{4\pi x_2} \int\!\rmd y\, \log\frac {1+|x-y|}{|x-y|} \omega_\eps(y,t) \,.
\]
To determine the behavior of $Q_\eps(t)$ as $\eps\to 0$,  we decompose,
\[
Q_\eps(t) = Q_\eps^1(t) + Q_\eps^2(t)\,,
\]
with 
\[
\begin{split}
Q_\eps^1(t) & :=  |\log\eps| \int_{\Sigma(q_\eps(t), \eps|\log\eps|)}\!\rmd x\, \omega_\eps(x,t) \\ & \quad \times  \frac{1}{4\pi x_2} \int_{\Sigma(q_\eps(t), \eps|\log\eps|)}\!\rmd y\, \log\frac {1+|x-y|}{|x-y|} \omega_\eps(y,t)\,.
\end{split}
\]
The rest $Q_\eps^2(t) = Q_\eps(t) - Q_\eps^1(t)$ is the sum of three terms, each one is the integration of the same function, which in view of Eq.~\eqref{lrest} is bounded by 
\[
\mc G(x,y) := \frac{1}{2\pi r_0} \log\frac {1+|x-y|}{|x-y|} \omega_\eps(x,t)\omega_\eps(y,t)\,,
\]
and in the integration domain at least one between the $x$ and the $y$ variable is contained in the set $\Sigma(q_\eps(t), \eps|\log\eps|)^\complement$. Therefore, since $\mc G$ is a symmetric function, 
\[
\begin{split}
Q_\eps^2(t) & \le \frac{3|\log\eps|}{2\pi r_0} \int_{\Sigma(q_\eps(t),\eps|\log\eps|)^\complement}\!\rmd x\, \omega_\eps(x,t) \int\!\rmd y\, \log\frac {1+|x-y|}{|x-y|} \omega_\eps(y,t)  \\ & \le  \frac{C}{\log|\log\eps|} \,,
\end{split}
\]
where we first bounded the $\rmd y$-integral as done in Eq.~\eqref{311b}, and then we used   Eq.~\eqref{lem1b}.

Concerning $Q_\eps^1(t)$, we can obtain a lower bound for it, by inserting a lower bound to the function $\frac{1}{4\pi x_2}\log\frac {1+|x-y|}{|x-y|}$ in the domain of integration and applying again Eq.~\eqref{lem1b}, 
\begin{equation}
\label{q2}
\begin{split}
	Q_\eps^1(t) & \ge \frac{|\log\eps|}{4\pi (q_{\eps,2}(t)+\eps|\log\eps|)} \log\frac {1+2\eps|\log\eps|}{2\eps|\log\eps|} \bigg(\int_{\Sigma(q_{\eps}(t),\eps|\log\eps|)}\!\rmd x\, \omega_\eps(x,t)\bigg)^2 \\ & \ge  \frac{|\log\eps|}{4\pi (q_{\eps,2}(t)+\eps|\log\eps|)} \log\frac {1+2\eps|\log\eps|}{2\eps|\log\eps|} \frac{1}{|\log \eps|^2}\bigg( 1 - \frac{C_1}{\log|\log\eps|}\bigg)^2\,.
\end{split}
\end{equation}
On the other hand, by Eqs.~\eqref{omega_t} and \eqref{311b}, we can obtain an upper bound
for $Q_\eps^1(t)$,
\begin{equation}
\label{q3}
\begin{split}
	Q_\eps^1(t) & \le \frac{1}{4\pi (q_{\eps,2}(t)-\eps|\log\eps|)} \sup_x \int\!\rmd y\, \log\frac {1+|x-y|}{|x-y|} \omega_\eps(y,t) \\ & \le  \frac{3M}{2\eps^2 |\log\eps|(q_{\eps,2}(t)-\eps|\log\eps|)}  \bigg\{\frac{\bar\rho^2}{2} \log\frac{1+\bar\rho}{\bar\rho} - \frac 12 \int_0^{\bar\rho}\!\rmd \rho\, \frac{\rho}{1+\rho} \bigg\}\,, 
\end{split}
\end{equation}
with $\bar\rho$ such that $3\pi\bar\rho^2 M/(\eps^2|\log\eps|)= 1/|\log\eps|$. Now, in view of Eq.~\eqref{lem1b}, the disk $\Sigma(q_\eps(t), \eps|\log\eps|)$ must have non empty intersection with $\Lambda_\eps(t)$. Therefore, since we are assuming $\eps$ so small that $T^0_\eps=T$, from Eq.~\eqref{stimGa} we deduce that
\begin{equation}
\label{q2r0}
\max_{t\in[0,T]}|q_{\eps,2}(t)-r_0| \leq \frac{C}{|\log\eps|^k}+ \eps |\log\eps|\,.
\end{equation}

We conclude that the right-hand side in both Eqs.~\eqref{q2} and \eqref{q3} converges to $1/(4\pi r_0)$ as $\eps\to 0$, so that, in view of Eq.~\eqref{b12},
\begin{equation}
\label{b1b}
\lim_{\eps\to 0} \max_{t\in [0, T]} \bigg|B_{\eps,1}(t) - \bigg(z_0 +\frac{t}{4\pi r_0}\bigg) \bigg| = 0 \,,
\end{equation}
which, together with Eq.~\eqref{compact2}, proves Eq.~\eqref{bz} (recall we fixed $a=1$).
\end{proof}

From Eq.~\eqref{lem1b} and Proposition \ref{prop:2}, the proof of item (2) of Theorem \ref{thm:2} is completed if we show that
\begin{equation}
\label{q1b1}
\lim_{\eps\to 0} \sup_{t\in [0,T]} [B_\eps(t) - q_\eps(t)] = 0
\end{equation}
(actually, by Eq.~\eqref{q2r0}, the convergence of the second component is already known, but this does not shorten the proof). To this aim, we set $\Sigma_t = \Sigma(q_\eps(t), \eps|\log\eps|)$ and compute,
\[
\begin{split}
|B_\eps(t) - q_\eps(t)| & \le |\log\eps|\int\! \rmd x\, |x - q_\eps(t)| \omega_\eps(x,t) \\ & = |\log\eps|\int_{\Sigma_t}\! \rmd x\, |x - q_\eps(t)| \omega_\eps(x,t) +  |\log\eps|\int_{\Sigma_t^\complement}\! \rmd x\, |x - q_\eps(t)| \omega_\eps(x,t) \\ & \le  \eps |\log\eps| +  \frac{C_1}{\log|\log\eps|} |B_\eps(t) - q_\eps(t)| \\ & \quad +  |\log\eps|\int_{\Sigma_t^\complement}\! \rmd x\, |x - B_\eps(t)| \omega_\eps(x,t)\,.
\end{split}
\]
where we used Eq.~\eqref{lem1b}. Therefore, by assuming $\eps$ so small to have $2C_1 \le \log|\log\eps|$,
\[
\begin{split}
|B_\eps(t) - q_\eps(t)| & \le 2 \eps |\log\eps| + \\ & \quad + 2   \sqrt{|\log\eps|\int_{\Sigma_t^\complement}\! \rmd x\, \omega_\eps(x,t)} \sqrt{|\log\eps|\int_{\Sigma_t^\complement}\! \rmd x\,  |x - B_\eps(t)|^2 \omega_\eps(x,t)} \\ & \le  2 \eps |\log\eps| + 2  \sqrt{\frac{C_1}{\log|\log\eps|}} \sqrt{ |\log\eps| J_\eps(t)}\,,
\end{split}
\]
where we applied again Eq.~\eqref{lem1b}, the Cauchy-Schwarz inequality, and introduced the moment of inertia with respect to center of vorticity defined as 
\begin{equation}
\label{J}
J_\eps(t) = \int \!\rmd x\,  |x-B_\eps(t)|^2 \omega_\eps(x,t)\,.
\end{equation}

Now, we claim that
\begin{equation}
\label{Jst}
J_\eps(t)\le \frac{C}{|\log\eps|} \quad \forall\, t\in [0,T]\,,
\end{equation} 
from which Eq.~\eqref{q1b1} follows in view of the above estimate on $|B_\eps(t) - q_\eps(t)| $. To prove the claim, we compute the time derivative of $J_\eps(t)$, by using Eq.~\eqref{weqF},
\[
\dot J_\eps(t) = 2\int \!\rmd x\, \omega_\eps(x,t) \, (x-B_\eps(t)) \cdot (u(x,t)+F^\eps(x,t) - \dot B_\eps(t))  \,,
\]
so that, in view of Eq.~\eqref{bpunto},
\[
\begin{split}
\dot J_\eps(t) & =  2 \int \rmd x \, \omega_\eps(x,t) \left[ u(x,t)-|\log\eps|\int \rmd y \, \omega_\eps(y,t) \, u(y,t) \right] \cdot (x-B_\eps(t))  \\ & \quad + 2 \int \rmd x \, \omega_\eps(x,t) \left[ F^\eps(x,t)-|\log\eps|\int \rmd y \, \omega_\eps(y,t) \, F^\eps(y,t) \right] \cdot (x-B_\eps(t)) \,.
\end{split}
\]
We consider first  the term containing $F^\eps$ and note that, by definition of $B_\eps(t)$,
\[ 
\begin{split}
& \int \rmd x\, \omega_\eps(x,t)\, (x-B_\eps(t)) \cdot \int\! \rmd y \, \omega_\eps(y,t)\,  F^\eps (y,t)   = 0 \,, \\ &   \int\! \rmd x\,  \omega_\eps(x,t)\, (x-B_\eps(t)) \cdot F^\eps(B_\eps(t), t) =0 \,. 
\end{split}\]
We thus obtain,
\[
\begin{split} 
& 2 \left|\int\! \rmd x \, \omega_\eps(x,t) \left[ F^\eps(x,t)-|\log\eps|\int\!\rmd y\, \omega_\eps(y,t) \, F^\eps (y,t) \right] \cdot (x-B_\eps(t)) \right| \\ & \quad = 2 \left| \int\! \rmd x \, \omega_\eps(x,t)  \left[ F^\eps(x,t)- F^\eps (B_\eps(t),t) \right] \cdot (x-B_\eps(t)) \right| \\ & \quad \le 2\int\! \rmd x \, \omega_\eps(x,t)\, \frac{ L}{|\log\eps|} |x-B_\eps(t)|^2  \le  \frac{2L}{|\log\eps|} \, J_\eps(t)\,, 
\end{split}
\]
where, in the last line, we used Eq.~\eqref{2Lipsc}. For  the term containing $u$, we have analogously,
\[ 
\int\! \rmd x\, \omega_\eps(x,t) \, (x-B_\eps(t)) \cdot \int\! \rmd y \, \omega_\eps(y,t)\, u(y,t) =0\,. 
\]
Moreover, by the  antisymmetry of $K$ and using Eq.~\eqref{decom_u},
\begin{equation}
\label{antisym}
\int\! \rmd x \, \omega_\eps(x,t) \, \widetilde{u}(x,t) = \int\! \rmd x \int\! \rmd y \, \omega_\eps(x,t)\,\omega_\eps(y, t) \, K(x-y) = 0, 
\end{equation}
so that, as $(x-y) \cdot K(x-y) =0$,
\[
\begin{split} 
\int\! \rmd x \, \omega_\eps(x,t) \, x \cdot \widetilde{u}(x,t) = & \int\! \rmd x \int\! \rmd y \, \omega_\eps(x,t)\,\omega_\eps(y, t) \, x \cdot K(x-y) \\ = & \int\! \rmd x \int\! \rmd y \, \omega_\eps(x,t)\,\omega_\eps(y, t) \, y \cdot K(x-y)\,,
\end{split}
\]
which implies that also this integral is zero by the antisymmetry of $K$. Therefore,
\[
\begin{split}
& 2 \left| \,\int \rmd x \, \omega_\eps(x,t) \left[ u(x,t)-|\log\eps|\int \rmd y \, \omega_\eps(y,t) \, u(y,t) \right] \cdot (x-B_\eps(t)) \, \right|  \\ & \quad \le 2 \int\! \rmd x \, \omega_\eps(x,t) \left|\int\!\rmd y\, L(x,y)\, \omega_\eps(y,t) + \int\!\rmd y\, \mc R(x,y)\, \omega_\eps(y,t)  \right| \, |x-B_\eps(t)| \\ & \quad \le C \int\! \rmd x \, \omega_\eps(x,t) \, |x-B_\eps(t)| \le \frac{C}{|\log\eps|^{1/2}} \, \sqrt{J_\eps(t)}\,,
\end{split}
\]
where we have used Eq.~\eqref{bc_bound2} and Cauchy-Schwarz inequality. 

In conclusion,
\[ 
|\dot J_\eps(t)| \le \frac{2L}{|\log\eps|}\,J_\eps(t) +\frac{C}{|\log\eps|^{1/2}} \, \sqrt{J_\eps(t)}\,. 
\]
Recalling that the initial data imply $J_\eps(0) \le 4\eps^2$, this differential inequality implies Eq.~\eqref{Jst}. The proof of item (2) of Theorem \ref{thm:2} is thus completed.

\appendix 

\section{Proof of Lemma \ref{lem:4}}
\label{app:a}

In absence of external field, the proof of the concentration estimate Eq.~\eqref{lem1b} given in \cite{BCM00} is based on the conservation along the motion of the kinetic energy $E=  \frac 12 \int\!\rmd\bs\xi\, |\bs u (\bs\xi,t)|^2$, which in cylindrical coordinates $x=(x_1,x_2) = (z,r)$ takes the form
\[
E =  \frac 12 \int\! \rmd x \,  2\pi x_2 |u(x,t)|^2\,.
\]
More precisely, the assumptions on the initial vorticity, together with Eq.~\eqref{cons-omr} and the conservation of the quantities
\[ 
M_0  = \int\! \rmd x\, \omega_\eps(x,t)\,, \qquad M_2  = \int\! \rmd x \, x_2^2 \omega_\eps(x,t)
\]
allow to compute the asymptotic behavior as $\eps\to 0$ of the energy $E$, from which the desired concentration estimate is deduced. 

In the present case, since the vector field $\bs F^\eps = (F^\eps_z,F^\eps_r,F^\eps_\theta) := (F^\eps_1,F^\eps_2,0)$ has zero divergence, the conservation law of $M_0$ is still valid by Liouville's theorem and Eq.~\eqref{cons-omr},
\begin{equation}
\label{m0c}
M_0(t) = \frac{1}{2\pi} \int\! \rmd\bs\xi\, \frac{\omega_\eps(\bs\xi,t)}{r} =  \int\! \rmd\bs\xi_0 \, \frac{\omega_\eps(\bs\xi_0,0)}{r_0} = M_0(0)
\end{equation}
(above, the change of variables\ is $\bs\xi=\phi^t(\bs\xi_0)$ with $\phi^t$ the flow generated by $\dot{\bs\xi} = \bs u(\bs\xi,t) + \bs F^\eps(\bs\xi,t)$).

Concerning the variation of $M_2$, since $\omega_\eps(x,t)$ has compact support, we can apply Eq.~\eqref{weqF} with $f(x,t) = x_2^2$, so that  
\[
\dot M_2 = \int\! \rmd x \, \omega_\eps(x,t) \,2 x_2  F^\eps_2(x,t) ,
\]
which implies $|\dot M_2| \le 2C_F |\log\eps|^{-3/2} \sqrt{M_2}$ in view of Assumption \ref{ass:1}, item (b). Therefore, using also Eq.~\eqref{initialF}, 
\[
M_2 \le 2(|\zeta^0_2| + \eps)^2 M_0 + 32C_F^2T^2 |\log\eps|^{-3} \le \const |\log\eps|^{-1}\,.
\] 
This is the same estimate, but for a larger constant, which is obtained in absence of $F^\eps$. On the other hand, the particular value of this constant is easily seen to be irrelevant in the proof of \cite[Thm.~1]{BCM00}.

Concerning the variation of $E$, we observe that in the proof of \cite[Thm.~1]{BCM00}  the energy conservation is used to deduce a lower bound of the form $E>C^*/|\log\eps| + \const /|\log\eps|^2$, because such estimate can be easily shown to be true for the energy computed at time zero. The crucial point in the argument is that $C^*$ is the same constant appearing in the leading term of an upper bound deduced for the energy $E$ at any time. Therefore, the proof given in \cite{BCM00} is valid also in the present case provided $|\dot E| \le \const/|\log\eps|^2$.   

To compute $\dot E$, we introduce the stream function
\[
\Psi(x,t) = \int\!\rmd y\, S(x,y)\, \omega_\eps(y,t)\,, 
\]
where the Green function $S(x,y)$ reads
\[
S(x,y) := \frac{x_2y_2}{2\pi} \int_0^\pi\!\rmd\theta\, \frac{\cos\theta}{\sqrt{|x-y|^2 + 2x_2y_2(1-\cos\theta)}}\,,
\]
so that $u(x,t) = x_2^{-1} \nabla^\perp\Psi(x,t)$ and the energy takes the form (see, e.g., \cite{BCM00,F})
\[
E = \pi \int\! \rmd x\, \Psi(x,t) \, \omega_\eps(x,t) = \pi \int\!\rmd x \int\!\rmd y\, S(x,y)\, \omega_\eps(x,t) \, \omega_\eps(y,t)\,.
\]
As before, since $\omega_\eps(x,t)$ has compact support we can apply Eq.~\eqref{weqF}, so that, since $u\cdot \nabla \Psi=0$,
\[
\dot E = \pi \int\!\rmd x\, \omega_\eps(x,t) \, (F^\eps \cdot \nabla \Psi + \partial_t \Psi)(x,t)\,,
\]
where, again from Eq.~\eqref{weqF},
\[
\partial_t\Psi(x,t) = \int\!\rmd y\, \omega_\eps(y,t)\, (u+F^\eps)(y,t) \cdot \nabla_y S(x,y)\,.
\]
Since $S(x,y)$ is symmetric then $\Psi(y,t) = \int\!\rmd x\, S(x,y)\, \omega_\eps(x,t)$, whence
\[
\begin{split}
\int\!\rmd x\, \omega_\eps(x,t)\, \partial_t \Psi (x,t) & = \int\!\rmd x\, \omega_\eps(x,t) \int\!\rmd y\, \omega_\eps(y,t)\, (u+F^\eps)(y,t) \cdot \nabla_y S(x,y) \\ & = \int\!\rmd y\, \omega_\eps(y,t)\, (u+F^\eps)(y,t) \cdot \nabla_y \int\!\rmd x\,S(x,y)\, \omega_\eps(x,t)\\ &  = \int\!\rmd y\, \omega_\eps(y,t) \, (F^\eps \cdot \nabla \Psi)(y,t)\,, 
\end{split} 
\]
where we used again the orthogonality condition $u\cdot \nabla \Psi=0$. In conclusion,
\[
\begin{split}
\dot E & = 2 \pi \int\!\rmd x\, \omega_\eps(x,t) (F^\eps \cdot \nabla \Psi)(x,t) \\ & = \ 2\pi \int\!\rmd x \!\int\! \rmd y\, \omega_\eps(x,t)\, F^\eps(x,t)\cdot  \nabla_xS(x,y)\, \omega_\eps(y,t)\,.
\end{split}
\]
Now, by direct computation (or simply recalling that $\nabla^\perp\Psi(x,t) = x_2 u(x,t)$ together with Eqs.~\eqref{u=}, \eqref{H1}, and \eqref{H2}), 
\[
\begin{split}
\frac{\partial S}{\partial x_1} & = \frac{x_2y_2}{2\pi} \int_0^\pi\!\rmd\theta\, \frac{(y_1-x_1)\cos\theta}{[|x-y|^2 + 2x_2y_2(1-\cos\theta)]^{3/2}}\,, \\ \frac{\partial S}{\partial x_2} & = \frac{x_2y_2}{2\pi} \int_0^\pi\!\rmd\theta\, \frac{y_2-x_2\cos\theta}{[|x-y|^2 + 2x_2y_2(1-\cos\theta)]^{3/2}}\,.
\end{split}
\]
Therefore, letting
\[
a := \frac{|x-y|}{\sqrt{x_2y_2}}
\]
and introducing the integrals 
\[
I_1(a) = \int_0^\pi\!\rmd\theta\, \frac{\cos\theta}{[a^2+2(1-\cos\theta)]^{3/2}}\, , \qquad I_2(a) = \int_0^\pi\!\rmd\theta\, \frac{1-\cos\theta}{[a^2+2(1-\cos\theta)]^{3/2}}\,,
\]
the gradient $\nabla_x S$ reads
\[
\nabla_xS(x,y) = \frac{I_1(a)(y-x)}{2\pi\sqrt{x_2y_2}}  + \frac{I_2(a)}{2\pi}\sqrt{\frac{y_2}{x_2}} \begin{pmatrix} 0 \\ 1 \end{pmatrix} =: A(x,y) + B(x,y),
\]
with $A(x,y)$ anti-symmetric, so that
\[
\begin{split}
\dot E & = \pi \int\!\rmd x \!\int\! \rmd y\, \omega_\eps(x,t) [F^\eps(x,t) - F^\eps(y,t)] \cdot A (x,y) \, \omega_\eps(y,t) \\ & \quad + 2\pi \int\!\rmd x \!\int\! \rmd y\, \omega_\eps(x,t) F^\eps(x,t)\cdot B(x,y) \, \omega_\eps(y,t)\,.
\end{split}
\]
In view of Assumption \ref{ass:1}, item (b), and noticing that $\div(x_2 F^\eps) =0$ implies $F^\eps_2((x_1,0),t)=0$, we deduce that
\[
\begin{split}
& \big|(F^\eps(x,t) - F^\eps(y,t)) \cdot A (x,y)\big| \le \frac{L}{2\pi|\log\eps|} I_1(a) \frac{|y-x|^2 }{\sqrt{x_2y_2}} \,, \\ & \big| F^\eps(x,t)\cdot B(x,y) \big| = \big| (F^\eps(x,t) - F^\eps(x_1,0,t)) \cdot B(x,y) \big| \le \frac{L}{2\pi|\log\eps|} I_2(a) \sqrt{x_2y_2}\,. 
\end{split}
\]
We now recall the following estimates, see, e.g., \cite[Appendix]{Mar99},
\[
\begin{split}
I_1(a) & \le \int_0^\pi\!\rmd\theta\, \frac{\cos\frac\theta 2}{[a^2+2(1-\cos\theta)]^{3/2}} = \int_0^\pi\!\rmd\theta\, \frac{\cos\frac\theta 2}{[a^2+4\sin^2\frac\theta 2]^{3/2}} = \frac{2}{a^2\sqrt{a^2+4}}\,, \\ I_2(a) & \le \frac 12\log\left(2+\sqrt{a^2+4}\right) - \frac 12 \log a + \const \\ & = \frac 12 \log\left(2\sqrt{x_2y_2} + \sqrt{b+4x_2y_2}\right) -\frac 14 \log b + \const\,,
\end{split}
\]
where in the last identity we introduced the quantity $b=|x-y|^2$ as in \cite{BCM00}. Therefore,
\[
\big| (F^\eps(x,t) - F^\eps(y,t)) \cdot A (x,y)\big| \le  \frac{\const}{|\log\eps|}\sqrt{x_2y_2}\,,
\]
\[
\begin{split}
\big| F^\eps(x,t)\cdot B(x,y) \big| & \le \frac{\const}{|\log\eps|} \\ & \times \sqrt{x_2y_2} \left[ \log\left(2\sqrt{x_2y_2} + \sqrt{b+4x_2y_2}\right) -\frac 12 \log b + \const \right]\,,
\end{split}
\]
so that 
\[
\begin{split}
|\dot E| & \le \frac{\const}{|\log\eps|} \int\!\rmd x \!\int\! \rmd y\, \omega_\eps(x,t) \omega_\eps(y,t) \\ & \quad \times \sqrt{x_2y_2} \left\{\const + \log\left(2\sqrt{x_2y_2} + \sqrt{b+4x_2y_2}\right) -\frac 12 \log b \right\}\,.  
\end{split}
\]
Except for the prefactor $\const/|\log\eps|$, the right-hand side in the above inequality equals the first upper bound on $E$ appearing in \cite[Eq.\ (2.30)]{BCM00}, from which the estimate $E \le \const/|\log\eps|$ is deduced. We conclude that $|\dot E| \le \const / |\log\eps|^2$ as required.
\qed


\begin{thebibliography}{99}

\bibitem{AmS89} Ambrosetti, A., Struwe, M.:  Existence of steady rings in an ideal fluid. Arch. Ration. Mech. Anal. {\bf 108}, 97--108 (1989)

\bibitem{BCM00} Benedetto, D., Caglioti, E., Marchioro, C.: On the motion of a vortex ring with a sharply concentrate vorticity. Math. Meth. Appl. Sci. \textbf{23}, 147--168 (2000)

\bibitem{BuM1} Butt\`a, P., Marchioro, C.: Long time evolution of concentrated Euler flows with planar symmetry. SIAM J. Math. Anal. \textbf{50}, 735--760 (2018)

\bibitem{BuM2} Butt\`a, P., Marchioro, C.: Time evolution of concentrated vortex rings. J. Math. Fluid Mech. \textbf{22}, Article number 19 (2020)

\bibitem{CS} Cetrone, D.,  Serafini, G.: Long time evolution of fuids with concentrated vorticity and convergence to the point-vortex model.  Rendiconti di Matematica e delle sue applicazioni \textbf{39}, 29--78 (2018)

\bibitem{CGP14} Colagrossi, A., Graziani, G., Pulvirenti, M: Particles for
fuids: SPH versus vortex methods. Mathematical and Mechanics of complex systems \textbf{2}, 45--70 (2014)

\bibitem{Fr70} Fraenkel, L.E.: On steady vortex rings of small cross-section in an ideal fluid. Proc. Roy. Soc. Lond. A. \textbf{316}, 29--62 (1970)

\bibitem{FrB74} Fraenkel, L.E., Berger M.S.: A global theory of steady vortex rings in an ideal fluid. Acta Math. \textbf{132}, 13--51 (1974)

\bibitem{F} Friedman, A.: Variational Principles and Free-Boundary Problems. Wiley, New York 1982

\bibitem{Gal11} Gallay, T.: Interaction of vortices in weakly viscous planar flows. Arch. Ration. Mech. Anal. \textbf{200}, 445--490 (2011)

\bibitem{MaB02} Majda, A., Bertozzi, A.: Vorticity and Incompressible Flow. Cambridge Texts in Applied Mathematics, Cambridge University Press, Cambridge, 2002

\bibitem{Mar90} Marchioro, C.: On the vanishing viscosity limit for two-dimensional Navier-Stokes equations with singular initial data. Math. Meth. Appl. Sci. \textbf{12},  463--470 (1990)

\bibitem{Mar98} Marchioro, C.: On the inviscid limit for a fluid with a concentrated vorticity.
Commun. Math. Phys. \textbf{196}, 53--65 (1998)

\bibitem{Mar99} Marchioro, C.: Large smoke rings with concentrated vorticity. Journ. Math. Phys. \textbf{40}, 869--883 (1999)

\bibitem{Mar07} Marchioro, C.: Vanishing viscosity limit for an incompressible fluid with concentrated vorticity. J. Math. Phys. \textbf{48}, 065302, 16 pp. (2007)

\bibitem{MaP94} Marchioro, C., Pulvirenti, M.: Mathematical theory of incompressible non-viscous fluids. Applied mathematical sciences vol.~96, Springer-Verlag, New York, 1994

\bibitem{ShL92} Shariff, K., Leonard:, A.: Vortex Rings. Annu. Rev. Fluid  Mech. \textbf{24}, 235--279 (1972)

\end{thebibliography}
\end{document}